%% file: patey.tex
\documentclass{amsart}

\usepackage[utf8]{inputenc}
\usepackage[charter]{mathdesign}

\usepackage{graphicx}
\usepackage{epsfig}
\usepackage{times}
\usepackage{mathptmx}
\usepackage{amsmath}
\usepackage{dcolumn}
\usepackage{multicol}

\usepackage{xspace}
\usepackage{wrapfig}
\usepackage{todonotes}
\usepackage{hyperref}

\newtheorem{theorem}{Theorem}[section]
 
\newtheorem{lemma}[theorem]{Lemma}
\newtheorem{corollary}[theorem]{Corollary}
\newtheorem{example}[theorem]{Example}
\newtheorem*{claim}{Claim}

\theoremstyle{definition}
\newtheorem{definition}[theorem]{Definition}

\newtheorem{question}[theorem]{Question}

\theoremstyle{remark}

\newcommand{\eqdef}{\stackrel{def}{=}}

\newcommand{\la}{\langle}
\newcommand{\ra}{\rangle}

\newcommand{\imp}{\rightarrow}

\newcommand{\biimp}{\leftrightarrow}
\newcommand{\Nb}{\mathbb{N}}

\newcommand{\Rcal}{\mathcal{R}}

\renewcommand{\setminus}{\smallsetminus}
\newcommand{\Bool}{\mathbb{B}}

\newcommand{\true}{\textup{\texttt{T}}}
\newcommand{\false}{\textup{\texttt{F}}}

\newcommand{\p}[1]{\left( #1 \right)}
\newcommand{\set}[1]{\left\{ #1 \right\}}

\newcommand{\tuple}[1]{\left\langle #1 \right\rangle}

\newcommand{\s}[1]{\ensuremath{{\sf{#1}}}}
\newcommand{\rca}{\s{RCA}_0}
\newcommand{\wkl}{\s{WKL}_0}
\newcommand{\wwkl}{\s{WWKL}_0}
\newcommand{\dnr}{\s{DNR}}

\newcommand{\rkl}{\s{RKL}}
\newcommand{\rwkl}{\s{RWKL}}

\newcommand{\srt}{\s{SRT}}
\newcommand{\sat}{\s{SAT}}
\newcommand{\issat}{\s{ISAT}}
\newcommand{\rsat}{\s{RSAT}}
\newcommand{\lrsat}{\s{LRSAT}}
\newcommand{\rsats}[1]{\s{RSAT}(#1)}
\newcommand{\lrsats}[1]{\s{LRSAT}(#1)}
\newcommand{\sats}[1]{\s{SAT}(#1)}
\newcommand{\opvars}[1]{\s{Var}(#1)}
\newcommand{\opass}[1]{\s{Assign}(#1)}

\newcommand{\lrsatzerovalid}{\lrsats{\false\mbox{-valid}}}

\newcommand{\lrsatonevalid}{\lrsats{\true\mbox{-valid}}}

\newcommand{\lrsathorn}{\lrsats{\mbox{Horn}}}

\newcommand{\lrsatcohorn}{\lrsats{\mbox{CoHorn}}}
\newcommand{\rsataffine}{\rsats{\mbox{Affine}}}
\newcommand{\lrsataffine}{\lrsats{\mbox{Affine}}}
\newcommand{\rsatbijunctive}{\rsats{\mbox{Bijunctive}}}
\newcommand{\lrsatbijunctive}{\lrsats{\mbox{Bijunctive}}}
\newcommand{\scolor}{\s{COLOR}}
\newcommand{\rcolor}{\s{RCOLOR}}
\newcommand{\lrcolor}{\s{LRCOLOR}}
\newcommand{\poly}[1]{\s{Pol}\p{#1}}
\newcommand{\inv}[1]{\s{Inv}\p{#1}}
\newcommand{\coclone}[1]{\tuple{#1}}

\hypersetup{
	pdfborder={0 0 0}
}

\usepackage{xcolor}	
\usepackage{soul}

\definecolor{lightblue}{rgb}{.60,.60,1}

\definecolor{lightred}{rgb}{1,.60,.60}

\title{The complexity of satisfaction problems\\ in reverse mathematics}
\author{
  Ludovic Patey
}

\address{
Laboratoire PPS, Universit\'e Paris Diderot, Paris, FRANCE}
\email{ludovic.patey@computability.fr}

\date{\today}

\begin{document}

\maketitle

\begin{abstract}
Satisfiability problems play a central role in computer science and engineering
as a general framework for studying the complexity of various problems.
Schaefer proved in 1978 that truth satisfaction of propositional formulas
given a language of relations is either NP-complete or tractable.
We classify the corresponding satisfying assignment construction problems
in the framework of reverse mathematics and show that the principles are either provable
over $\rca$ or equivalent to $\wkl$. We formulate also a Ramseyan version of the problems
and state a different dichotomy theorem. However, the different classes arising from this classification
are not known to be distinct.
\end{abstract}

\section{Introduction}
\input{parts/introduction-long}

\section{Schaefer's dichotomy theorem}
\input{parts/schaefer-long}

\section{Ramsey-type Schaefer's dichotomy theorem}
\input{parts/ramsey-schaefer-long}

\section{The strength of satisfiability}
\input{parts/variants-long}

\section{Conclusions and questions}
\input{parts/conclusion-long}

\vspace*{0.5cm}

\noindent \textbf{Acknowledgements}.
The author is thankful to Laurent Bienvenu and Paul Shafer
for their availability during the different steps giving birth to the paper,
and for their useful suggestions.
The author is funded by the John Templeton Foundation (`Structure and Randomness in the Theory of Computation' project). 
The opinions expressed in this publication are those of the author(s) and do not necessarily reflect the views of the John Templeton Foundation.

\bibliography{patey}
\bibliographystyle{plain}

\clearpage
\appendix
\input{parts/appendix-long}

\end{document}

%% file: parts/introduction-long.tex

A common way to solve a constrained problem in industry consists
in reducing it to a satisfaction problem over propositional logic and using 
a SAT solver. The generality of the framework and its multiple applications 
make it a natural subject of interest for the scientific community and constraint satisfaction problems remains
an active field of research.

In 1978, Schaefer \cite{schaefer1978complexity} gave a great insight in the understanding of the complexity
of satisfiability problems by studying a parameterized class of problems
and showing they admit a dichotomy between NP-completeness and tractability.
Many other dichotomy theorems have been proven since, about refinements to~$AC^0$ reductions \cite{allender2005complexity},
variants about counting, optimization, 3-valued domains and many others 
\cite{creignou1996complexity,khanna1996optimization,bulatov2002dichotomy}.
The existence of dichotomies for $n$-valued domains with $n > 3$ remains open.

Reverse mathematics is a vast program for the classification of the strength of mathematical theorems. 
It uses proof theoretic methods to reveal the computational content of theorems. This study has led to the main observation
that many theorems are computationally equivalent to one of four axioms. One particular axiom
is Weak König's lemma ($\wkl$) which allows formalization of many compactness arguments 
and the solution to many satisfiability problems.
We believe that studying constraint satisfaction problems (CSP) within this framework
can lead to insights in both fields: in reverse mathematics, we can exploit the generality of constraint satisfaction
problems to compare existing principles by reducing them to satisfaction problems.
In CSP, reverse mathematics can yield a better understanding of the computational strength of satisfiability problems for 
particular classes of formulas. In particular we answer the question of Marek \& Remmel~\cite{marek2009complexity}
whether there are dichotomy theorems for infinite recursive versions of constraint satisfaction problems.\footnote{
This paper is an extended version of a conference paper of the same name published in CiE 2014.}

\begin{definition}
Let $\Bool = \{\false, \true\}$ be the set of Booleans.
An (infinite) set of Boolean formulas $C$ is \emph{finitely satisfiable} if every conjunction of a finite set of formulas in $C$
is satisfiable. 
$\sat$ is the statement  
``For every finitely satisfiable set $C$ of Boolean formulas over an infinite set of variables $V$, 
there exists an infinite assignment  $\nu: V \to \Bool$ satisfying $C$.''
The pair $(V, C)$ forms an \emph{instance} of~$\sat$.
\end{definition}

The base axiom system for reverse mathematics is called $\rca$, standing for Recursive Comprehension Axiom.
It consists of basic Peano axioms together with a comprehension scheme restricted to~$\Delta^0_1$
formulas and an the induction restricted to~$\Sigma^0_1$ formulas. 

\begin{theorem}[Simpson~\cite{simpson2009subsystems}]\label{thm:wkl-sat}
$\rca \vdash \wkl \biimp \sat$
\end{theorem}
\begin{proof}
$\wkl \imp \sat$: Let $C$ be a finitely satisfiable set of formulas over a set of variables $V$.
Let $\la x_i \mid i \in \Nb \ra$ enumerate $V$.  For each $\sigma \in 2^{<\Nb}$, identify $\sigma$ with the truth assignment $\nu_\sigma$ on $\{x_i \mid i < |\sigma|\}$ given by $(\forall i < |\sigma|)(\nu_\sigma(x_i) = \true \biimp \sigma(i) = 1)$.  Let $T \subseteq 2^{<\Nb}$ be the tree
$T = \{\sigma \in 2^{<\Nb} \mid \neg(\exists \theta \in C \restriction |\sigma|)(\nu_\sigma(\theta) = \false)\}$,
where $C \restriction |\sigma|$ is the set of formulas in $C$ coded by numbers less than $|\sigma|$, and $\nu_\sigma(\theta)$ is the truth value assigned to~$\theta$ by $\nu_\sigma$ (note that $\nu_\sigma(\theta)$ is undefined if $\theta$ contains a variable $x_m$ for $m \geq |\sigma|$).  $T$ exists by $\Delta^0_1$ comprehension and is downward closed.  $T$ is infinite because for any $n \in \Nb$, any satisfying truth assignment of~$C \restriction n$ restricted to~$\{x_i \mid i<n\}$ yields a string in $T$ of length $n$.  By $\wkl$
 let $P \subseteq \Nb$ be a path through $T$. We show that every finite $C_0 \subseteq C$ can be satisfied by the truth assignment $\nu : V \to \Bool$ defined for all $x_i \in V$ by $\nu(x_i) = \true \biimp i \in P$. Given $C_0 \subseteq C$ finite, let $n$ be such that $C_0 \subseteq C \restriction n$ and such that $\opvars{C_0} \subseteq \{x_i \mid i < n\}$.  Now let $\sigma \prec P$ be such that $|\sigma| = n$.  Then $(\forall \theta \in C_0)(\nu_\sigma(\theta) = \true)$ because $\nu_\sigma(\theta)$ is defined for all $\theta \in C_0$ and $\nu_\sigma(\theta) \neq \false$ for all $\theta \in C_0$.  Thus $\nu_\sigma$ satisfies $C_0$.

$\sat \imp \wkl$:  Let $V = \{x_i \mid i \in \Nb\}$ be a set of distinct variables, and to each string $\sigma \in 2^{<\Nb}$, associate the formula $\theta_\sigma \equiv \bigwedge_{i<|\sigma|}\ell_i$, where $\ell_i \equiv x_i$ if $\sigma(i) = 1$ and $\ell_i \equiv \neg x_i$ if $\sigma(i) = 0$.  Let $T \subseteq 2^{<\Nb}$ be an infinite tree, and, for each $n \in \Nb$, let $T^n = \{\sigma \in T \mid |\sigma| = n\}$.  Let $C = \{\bigvee_{\sigma \in T^n} \theta_\sigma \mid n \in \Nb\}$.  We show that every finite $C_0 \subseteq C$ is satisfiable.  Given $C_0 \subseteq C$ finite, let $n$ be maximum such that $\bigvee_{\sigma \in T^n}\theta_\sigma \in C_0$ and, as $T$ is infinite, let $\tau \in T$ have length $n$.  Then $\theta_\tau \imp \phi$ for every $\phi \in C_0$ because if $\phi =\bigvee_{\sigma \in T^m} \theta_\sigma \in C_0$, then $m \leq n$, $\theta_{\tau \restriction m}$ is a disjunct of~$\phi$, and $\theta_\tau \imp \theta_{\tau \restriction m}$.  Therefore $C_0$ is satisfiable by the truth assignment that satisfies $\theta_\tau$.  By $\sat$ there exists a valid assignment $\nu$ for $C$. Let $P$ be $\set{i \in \Nb : \nu(x_i) = \true}$. We show that $P$ is a path through $T$.
Given $n \in \Nb$, let $\sigma \prec P$ be such that $|\sigma| = n$. By definition of $P$, $(\forall i< n)(\sigma(i)=1 \biimp \nu(x_i) = \true)$, so $\nu(\theta_\sigma) = \true$, from which it follows that $\sigma \in T^n$.
\end{proof}

$\rwkl$, a weakening of~$\wkl$, has been recently introduced by Flood in \cite{flood2012reverse}.
Given an infinite binary tree, the principle does not assert the existence of a path, 
but rather of an infinite subset of a path through the tree.
Initially called $\rkl$, it has been renamed to~$\rwkl$ in \cite{bienvenurwkl} to give a consistent $R$ prefix to Ramseyan principles. 
This principle has been shown to be strictly weaker than $\srt^2_2$ and $\wkl$ by Flood, and strictly stronger than $\dnr$
by Bienvenu \& al. in \cite{bienvenurwkl}.
By analogy with $\rwkl$, we formulate Ramsey-type versions of satisfiability problems.

\begin{definition}
Let $C$ be a sequence of Boolean formulas over an infinite set of variables $V$.
A set $H$ is {\itshape homogeneous for $C$} if there is a truth value $c \in \Bool$ 
such that every conjunction of a finite set of formulas in $C$ 
is satisfiable by a truth assignment $\nu$ such that $(\forall a \in H)(\nu(a) = c)$.
\end{definition}

\begin{definition}
$\lrsat$ is the statement ``Let $C$  be a finitely satisfiable set of Boolean formulas over an infinite set of variables $V$.
For every infinite set $L \subseteq V$ there exists an infinite set $H \subseteq L$ homogeneous for $C$.''
The corresponding \emph{instance} of~$\lrsat$ is the tuple $(V, C, L)$. 
$\rsat$ is obtained by restricting $\lrsat$ to~$L = V$.
Then an \emph{instance} of~$\rsat$ is an ordered pair $(V, C)$.
\end{definition}

The equivalence between $\wkl$ and $\sat$ over $\rca$ extends to their Ramseyan version.
The proof is relatively easy and directly adaptable from proof of Theorem~\ref{thm:wkl-sat}.

\begin{theorem}[Bienvenu \& al. \cite{bienvenurwkl}]
$\rca \vdash \rwkl \biimp \rsat \biimp \lrsat$
\end{theorem}

\subsection{Definitions and notations}

Some classes of Boolean formulas -- bijunctive, affine, horn, ... --
have been extensively studied in complexity theory, leading to
the well-known dichotomy theorem due to Schaefer. We give
a precise definition of those classes in order to state our dichotomy theorems.

\begin{definition}
A \emph{literal} is either a Boolean variable (positive literal), or its negation (negative literal).
A \emph{clause} is a disjunction of literals.
A clause is \emph{horn} if it has at most one positive literal,
\emph{co-horn} if it has at most one negative literal and \emph{bijunctive} if it has at most 2 literals.
If we number Boolean variables, we can associate to each Boolean formula $\varphi$
with Boolean variables $x_1, \dots, x_n$ a relation $[\varphi] \subseteq \Bool^n$
such that $\vec a \in [\varphi]$ iff $\varphi(\vec a)$ holds. If $S$ is a set of relations,
an \emph{$S$-formula} over a set of variables $V$ is a formula of the form $R(y_1, \dots, y_n)$
for some $R \in S$ and $y_1, \dots, y_n \in V$.
\end{definition}

\begin{example}
Let $S = \set{\imp}$. $(x \imp y)$ is an $S$-formula but $(x \imp \neg y)$ is not. 
Neither is $(x \imp y) \wedge (y \imp z)$. The formula $(x \imp y)$ is equivalent to the horn
clause $(\neg x \vee y)$ where the literals are $\neg x$ and $y$.
\end{example}

\begin{definition}
A formula $\varphi$ is \emph{i-valid} for $i \in \Bool$ if $\varphi(i, \dots, i)$ holds.
It is \emph{horn} (resp. \emph{co-horn}, \emph{bijunctive}) if it is a conjunction of 
horn (resp. co-horn, bijunctive) clauses. A formula is \emph{affine}
if it is a conjunction of formulas of the form $x_1 \oplus \dots \oplus x_n = i$ for $i \in \Bool$
where $\oplus$ is the exclusive or.
\end{definition}

A relation $R  \subseteq \set{0, 1}^n$ is \emph{bijunctive} (resp. \emph{horn}, \emph{co-horn}, \emph{affine},
\emph{$i$-valid}) if there is bijunctive (resp. horn, co-horn, affine, $i$-valid) formula $\varphi$ such that $R = [\varphi]$.
A relation $R$ is \emph{i-default} for $i \in \Bool$ if for every $\vec{r} \in R$ and every $j < |\vec{r}|$,
the vector $\vec{s}$ defined by $\vec{s}(j) = i$ and $\vec{s}(k) = \vec{r}(k)$ otherwise, is also in $R$.
In particular every $i$-default relation is $i$-valid.
We denote by $\issat(S)$ the class of satisfiable conjunctions of~$S$-formulas.

\subsection{Dichotomies}

We first state the celebrated dichotomy theorem from Schaefer.
Interestingly, the corresponding dichotomies in reverse mathematics are not based on the same classes
of relations as the ones from Schaefer.

\begin{theorem}[Schaefer's dichotomy \cite{schaefer1978complexity}]
Let $S$ be a finite set of Boolean relations. If $S$
satisfies one of the conditions $(a)-(f)$ below, then
$\issat(S)$ is polynomial-time decidable. Otherwise,
$\issat(S)$ is log-complete in NP.
\smallskip

\begin{minipage}[t]{0.5\linewidth}  
\begin{itemize}
  \item[(a)] Every relation in S is $\false$-valid.
  \item[(b)] Every relation in S is $\true$-valid.
  \item[(c)] Every relation in S is horn
\end{itemize}
\end{minipage}
\begin{minipage}[t]{0.5\linewidth} 
\begin{itemize}
  \item[(d)] Every relation in S is co-horn
  \item[(e)] Every relation in S is affine.
  \item[(f)] Every relation in S is bijunctive.
\end{itemize}
\end{minipage}
\end{theorem}

In the remainder of this paper, $S$ will be a -- possibly infinite -- class of Boolean relations.
Note that there is no effectiveness requirement on $S$.

\begin{definition}
$\sats{S}$ is the statement
``For every finitely satisfiable set $C$ of~$S$-formulas over an infinite set of variables $V$,
there exists an infinite assignment 
$\nu: V \to \Bool$ satisfying $C$''.
\end{definition}

We will prove the following theorem
based on Schaefer's theorem.

\begin{theorem}\label{thm:inf-schaefer-dichotomy}
If $S$ satisfies one of the conditions $(a)-(d)$ below,
then $\sats{S}$ is provable over $\rca$. Otherwise $\sats{S}$
is equivalent to~$\wkl$ over $\rca$.
\begin{itemize}
  \item[(a)] Every relation in S is $\false$-valid.
  \item[(b)] Every relation in S is $\true$-valid.
  \item[(c)] If $R \in S$ is not $\false$-default then $R = [x]$.
  \item[(d)] If $R \in S$ is not $\true$-default then $R = [\neg x]$.
\end{itemize}
\end{theorem}

$\sats{S}$ principles are not fully satisfactory as these are not robust notions:
if we define $\sats{S}$ in terms of satisfiable sets of \emph{conjunctions} of~$S$-formulas,
this yields a different dichotomy theorems. 
In particular, $\rca \vdash \sats{[x], [\neg y]}$ whereas
$\rca \vdash \sats{[x \wedge \neg y]} \biimp \wkl$. Ramseyan versions of
satisfaction problems have better properties.

\begin{definition}
$\rsats{S}$ is the statement
``For every finitely satisfiable set $C$ of~$S$-formulas over an infinite set of variables $V$,
there exists an infinite set $H \subseteq V$ homogeneous for~$C$''.
\end{definition}

Usual reductions between satisfiability problems involve
fresh variable introductions. This is why it is natural
to define a \emph{localized} version of those principles,
i.e. where the homogeneous set has to lie within a pre-specified set.

\begin{definition}
$\lrsats{S}$  is the statement  
``For every finitely satisfiable set $C$ of~$S$-formulas over an infinite set of variables $V$ 
and every infinite set $X \subseteq V$,
there exists an infinite set $H \subseteq X$ homogeneous for~$C$''.
\end{definition}

In particular, we define $\lrsatzerovalid$ (resp. $\lrsatonevalid$, $\lrsathorn$,
$\lrsatcohorn$, $\lrsatbijunctive$ or $\lrsataffine$) to denote $\lrsats{S}$
where $S$ is the set of all $\false$-valid (resp. $\true$-valid, horn, co-horn, bijunctive or affine) relations.
We will prove the following dichotomy theorem.

\begin{theorem}\label{thm:dichotomy-ramsey-schaefer}
Either $\rca \vdash \lrsats{S}$ or $\lrsats{S}$ is equivalent to one of the following principles over $\rca$:

\begin{minipage}[t]{0.4\linewidth} 
\begin{itemize} 
  \item[1.] $\lrsat$
  \item[2.] $\lrsats{[x \neq y]}$
\end{itemize}
\end{minipage}
\begin{minipage}[t]{0.4\linewidth}  
\begin{itemize}
  \item[3.] $\lrsataffine$
  \item[4.] $\lrsatbijunctive$
\end{itemize}
\end{minipage}
\end{theorem}

As we will see in Theorem~\ref{thm:localized-to-unlocalized}, each of those 
principles are equivalent to their non localized version.
As well, $\lrsats{[x \neq y]}$ coincides with an already existing principle
about bipartite graphs~\cite{bienvenurwkl} called $\rcolor_2$ and $\lrsat$ is equivalent to~$\rwkl$
over $\rca$. Hence $\lrsats{S}$ is either provable over $\rca$, or equivalent 
to one of~$\rcolor_2$, $\rsataffine$, $\rsatbijunctive$
and $\rwkl$ over $\rca$.

%% file: parts/schaefer-long.tex
\begin{definition}
Let $S$ be a set of Boolean relations and $V$ be a set of variables. 
Let $\varphi$ be an $S$-formula over $V$. We denote by $\opvars{\varphi}$ the set of variables occurring in $\varphi$.
An \emph{assignment} for $\varphi$ is a function $\nu: \opvars{\varphi} \to \set{\true, \false}$.
An assignment can be naturally extended to a function over formulas by the natural interpretation rules
for logical connectives. Then an assignment $\nu$ \emph{satisfies} $\varphi$ if $\nu(\varphi) = \true$.
The set of assignments satisfying $\varphi$ is written $\opass{\varphi}$.
\emph{Variable substitution} is defined in the usual way and is written $\varphi[y/x]$, meaning that all occurrences
of $x$ in $\varphi$ are replaced by $y$. We will also write $\varphi[y/X]$ where $X$ is a set of variables
to denote substitution of all occurrences of a variable of $X$ in $\varphi$ by $y$.
A \emph{constant} is either $\false$ or $\true$.
\end{definition}

\begin{definition}
Let $S$ be a set of Boolean relations. 
The class of existentially quantified $S$-formulas
-- i.e. of the form $(\exists \vec{x})R[\vec{x}, \vec{y}]$ with $R \in S$ -- is denoted
by $Gen^{*}_{NC}(S)$. We also define $Rep^{*}_{NC}(S) = \set{[R] : R \in Gen^{*}_{NC}(S)}$, ie. the
relations represented by existentially quantified $S$-formula.
By abuse of notation, we may use $Rep^{*}_{NC}(R)$ when $R$ is a relation to denote $Rep^{*}_{NC}(\set{R})$.
\end{definition}

Given some set of Boolean realtions~$S$, the set~$Rep^{*}_{NC}(S)$
might not exist over~$\rca$. However, we shall not use it as a set, but within
relations of the form~$[R] \in Rep^{*}_{NC}(S)$, which can be seen as an abbreviation
for an arithmetical statement using only~$R$ and~$S$ as parameters.
Also note that the definition of $Gen^{*}_{NC}(S)$ and $Rep^{*}_{NC}(S)$ differ from
Schaefer's definition of $Gen_{NC}(S)$ and $Rep_{NC}(S)$ in that the latter are closed under conjunction.
Therefore, reusing Schaefer's lemmas must be done with some precautions,
checking that his proofs do not use conjunction. This is the case of the following lemma:

\begin{lemma}[{{Schaefer in \cite[4.3]{schaefer1978complexity}}}]\label{lem:schaefer-4.3}
$\rca$ proves that at least one of the following holds:
\begin{itemize}
  \item[(a)] Every relation in $S$ is $\false$-valid.
  \item[(b)] Every relation in $S$ is $\true$-valid.
  \item[(c)] $[x]$ and $[\neg x]$ are contained in $Rep^{*}_{NC}(S)$.
  \item[(d)] $[x \neq y] \in Rep^{*}_{NC}(S)$.
\end{itemize}
\end{lemma}

One easily sees that if every relation in $S$ is $\false$-valid (resp. $\true$-valid)
then $\rca \vdash \sats{S}$ as the assignment always equal to~$\false$ (resp. $\true$)
is a valid assignment and is computable. We will now see that
problems parameterized by relations either $\false$-default or $[x]$
(resp. $\true$-default or $[\neg x]$) are also solvable over~$\rca$. 

The proof of the following lemma justifies the name $\false$-default (resp. $\true$-default) 
by using a strategy for solving an instance $(V, C)$ of $\sats{S}$ consists in 
defining an assignment which given a variable $x$ will give it the default value $\false$ (resp. $\true$) 
unless it finds the clause $(x) \in C$, where $(x)$ is the clause with~$x$ as the unique literal.

\begin{lemma}\label{lem:i-valid-trivial}
$\rca$ proves that if the only relation in $S$ which is not $\false$-default is $[x]$ or the only relation which is not $\true$-default 
is $[\neg x]$ then $\sats{S}$ holds.
\end{lemma}
\begin{proof}
Assume $[x]$ is the only relation of $S$ which is not $\false$-default. 
Given an instance $(V, C)$ of $\sats{S}$, 
define the assignment $\nu : V \to \set{\false, \true}$ as follows:
$\nu(x) = \true$ iff $(x) \in C$. The assignment $\nu$ exists by $\Delta^0_1$-comprehension.
Suppose for the sake of contradiction that there is a formula $\varphi \in C$
such that $\nu(\varphi) = \false$. If $\varphi = (x)$ for some variable $x$, 
then by definition of $\nu$, $\nu(x) = \true$ hence $\nu(\varphi) = \true$.
So suppose $\varphi = R(x_1, \dots, x_n)$ for some $n \in \Nb$, where $R$ is a $\false$-default 
relation. Let $I = \set{i < n : (x_i) \in C}$. As $C$ is finitely satisfiable,
so is $\varphi \bigwedge_{i \in I} (x_i)$. Let $\mu$ be an assignment
satisfying $\varphi \bigwedge_{i \in I} (x_i)$.
In particular $\mu(x_i) = \true$ for each $i \in I$ and $\mu$ satisfies $\varphi$.
By $\false$-defaultness of $R$, the vector $\vec{r}$ defined by $\vec{r}(i) = \true$
for $i \in I$ and $\vec{r}(i) = \false$ otherwise is in $R$. But by definition of $\nu$,
$\nu(x_i) = \true$ iff $i \in I$, hence $\vec{r} = \nu(x_1)\dots \nu(x_n) \in R$
and $\nu(\varphi) = \true$.
So $\nu$ is a valid assignment and the proof can easily be formalized over $\rca$.
Hence $\rca \vdash \sats{S}$.
The same reasoning holds whenever the only relation of $S$ which is not $\true$-default is~$[\neg x]$.

\end{proof}

The following lemma simply reflects the fact that $\sats{[x \neq y]}$
can be seen as a reformulation of $\scolor_2$ which is equivalent to~$\wkl$ over $\rca$ \cite{hirst1990marriage}.

\begin{lemma}\label{lem:neq-is-wkl}
$\rca$ proves that if $[x \neq y] \in Rep^{*}_{NC}(S)$, then $\wkl \biimp \sats{S}$.
\end{lemma}
\begin{proof}
As $\rca \vdash \wkl \imp \sat$, it suffices to prove that
$\rca \vdash \sats{S} \imp \wkl$ to obtain desired equivalence.
Fix an infinite, locally bipartite, computable graph $G = (V, E)$ and 
let $\theta \in Gen^{*}_{NC}(S)$ be such that $[\theta] = [x \neq y]$.
By definition, $\theta = (\exists \vec{z})R(x, y, \vec{z})$ for some $R \in S$. 
Take an infinite set~$W$ of fresh variables disjoint from~$V$ and
define an instance $(V \cup W, C)$ of $\sats{S}$
by taking $C = \{ R(x, y, \vec{z}) : x < y \wedge \set{x, y} \in E \wedge (\vec{z} \in W \mbox{ has not yet been used})\}$.
The set $C$ is finitely satisfiable because~$G$ is locally bipartite.
Let $\nu : V \cup W \to \Bool$ be an assignment satisfying $C$
and let $P_0 = \{x \in V : \nu(x) = \false\}$ and $P_1 = \{x \in V : \nu(x) = \true\}$.
We claim that $P_0, P_1$ is a bipartition of~$G$. Suppose for the sake of absurd that the exists
an $i < 2$ and two elements $x < y \in P_i$ such that $\{x, y\} \in E$.
Then there exists fresh variables $\vec{z} \in W$ such that $R(x, y, \vec{z}) \in C$.
In particular, $\nu$ satisfies $R(x,y, \vec{z})$, hence the formula $\theta(x,y)$
so $\nu(x) \neq \nu(y)$, contradicting the assumption that $x, y \in P_i$.
Hence $\rca \vdash \sats{S} \imp \scolor_2$. 
\end{proof}

Theorem~\ref{thm:inf-schaefer-dichotomy} is proven by a case analysis using
Lemma~\ref{lem:schaefer-4.3}, by noticing that when we are not in cases
already handled by Lemma~\ref{lem:i-valid-trivial} and Lemma~\ref{lem:neq-is-wkl},
we can find $n$-ary formulas encoding $[x]$ and $[\neg x]$ with $n \geq 2$.
Thus diagonalizing against some values becomes a $\Sigma^0_1$ event.

\begin{proof}[Proof of Theorem~\ref{thm:inf-schaefer-dichotomy}]
We reason by case analysis. Cases where every relation in $S$ is $\false$-valid (resp. $\true$-valid) are trivial.
Cases where the only relation in $S$ which is not $\false$-default (resp. $\true$-default) is $[x]$ (resp. $[\neg x]$),
and whenever $[x \neq y] \in Rep^{*}_{NC}(S)$ are already handled by 
Lemma~\ref{lem:i-valid-trivial} and Lemma~\ref{lem:neq-is-wkl}.

In the remaining case, by Lemma~\ref{lem:schaefer-4.3}, $[x]$ and 
$[\neg x] \in Rep^{*}_{NC}(S)$. 
First we show that it suffices to find two relations $R_1, R_2 \in S$
together with two formulas $\psi_1, \psi_2 \in Gen^{*}_{NC}(S)$
such that $x_1 \not \in \opvars{\psi_1} \cup \opvars{\psi_2}$ and the following holds
$$
[(\exists \vec{z})R_1(x_1, \vec{z}) \wedge \psi_1(\vec{z})] = [x_1]
\hspace{10pt} \mbox{ and } \hspace{10pt} 
[(\exists \vec{z})R_2(x_1, \vec{z}) \wedge \psi_2(\vec{z})] = [\neg x_1]
$$
to prove the existence of a path through an infinite binary tree. Then, we show that such relations exist.
Note that the difference with the assumption that $[x]$ and 
$[\neg x] \in Rep^{*}_{NC}(S)$ is that the relations $R_1$ and $R_2$ have arity greater than 1,
hence the relations $R_1$ and $R_2$ may be added arbitrarily late to the set of formulas
with fresh variables.
Fix two disjoint sets of variables: $V = \set{x_\sigma : \sigma \in 2^{<\Nb}}$
and $W = \set{y_1, \dots}$. Let~$T \subseteq 2^{<\Nb}$ be an infinite tree.
We define an instance $(V \cup W, C)$
of $\sats{S}$ such that every satisfying assignment computes an infinite path through~$T$.
We define the set $C$ by stages $C_0 = \emptyset \subseteq C_1 \subseteq \dots$
Assume that at stage $s$, the existence of each $S$-formula
over variables $\set{x_{\sigma}, y_i : \sigma \in 2^i, i < s}$ has been decided. 
Given some string~$\sigma \in 2^{<\Nb}$, we denote by~$T^{[\sigma]}_s$
the set of strings~$\tau \in T$ of length~$s$ such that~$\tau \succeq \sigma$.
\begin{itemize}
	\item[1.] If $T^{[\sigma^\frown 0]}_s$ is empty but not~$T^{[\sigma^\frown 1]}_s$ for some~$\sigma \in 2^{<s}$,
then add $R_2(x_\sigma, \vec{y})$ and $\psi_2(\vec{y})$ to~$C_s$ for some fresh variables
$\vec{y} \in W \setminus \set{y_i : i < s}$.
	\item[2.] If $T^{[\sigma^\frown 1]}_s$ is empty but not~$T^{[\sigma^\frown 0]}_s$ for some~$\sigma \in 2^{<s}$,
then add $R_1(x_\sigma, \vec{y})$ and $\psi_1(\vec{y})$
to~$C_s$ for some fresh variables $\vec{y} \in W \setminus \set{y_i : i < s}$.
\end{itemize}
This finishes the construction.
We have ensured that for any satisfying assignment $\nu$ for $C$ and any string~$\sigma \in T$
inducing an infinite subtree, $\sigma^\frown \nu(x_\sigma)$ also induces an infinite subtree.
Define the strictly increasing sequence of strings~$\sigma_0 = \epsilon \prec \sigma_1 \prec \dots$
by~$\sigma_{s+1} = \sigma_s^\frown \nu(x_{\sigma_s})$.
The set~$P = \bigcup_s \sigma_s$ is an infinite path through~$T$.
This proof can easily be formalized in $\rca$. Hence $\rca \vdash \sats{S} \imp \wkl$.

We now find the relations $R_1, R_2 \in S$ and define the formulas $\psi_1$ and $\psi_2 \in Gen^{*}_{NC}(S)$.
Suppose there exists a relation $R_1 \in S$ which is not $\false$-valid and is different from $[x]$.
Define the formula
$\varphi = R_1(x_1, \dots)$ and let $\nu \in \opass{\varphi}$ be such that $\forall U \subseteq \nu^{-1}(\set{\true})$, 
the assignment which coincides with $\nu$ except for $U$ does not satisfy $\varphi$. 
Because $R_1$ is not $\false$-valid, $\nu^{-1}(\set{\true}) \neq \emptyset$. Suppose w.l.o.g. that $x_1 \in \nu^{-1}(\set{\true})$.
Then the following holds for some constants $i_2, i_3, \dots$
$$
[\varphi \bigwedge_{x \in \nu^{-1}(\set{\true}) \setminus \set{x_1}}(x)
\bigwedge_{x \in \nu^{-1}(\set{\false})}(\neg x)] = [x_1 \wedge (x_2 = i_2) \wedge (x_3 = i_3) \dots]
$$

Suppose now the only non $\false$-valid relation in $S$ is $[x]$, in which case there is a $\false$-valid relation $R_1 \in S$
which is not $\false$-default. Thus there is a non-empty finite set $I \subset \omega$ and a vector $\vec{r} \in R_1$
such that $\vec{r}(i) = \true$ for each $i \in I$, but for every such $\vec{r} \in R_1$,
$\exists j \not \in I$ such that $\vec{r}(j) = \true$. Consider a minimal (in pointwise natural order) such $\vec{r}$.
Define the formula $\varphi = R_1(x_1, \dots)$. Suppose without loss of generality that $1 \not \in I$ and $\vec{r}(1) = \true$.
Then the following holds for some constants $i_2, i_3, \dots$
$$
[\varphi \bigwedge_{i \in I}(x_i) \bigwedge_{\vec{r}(i) = 0}(\neg x_i)]
= [x_1 \wedge (x_2 = i_2) \wedge (x_3 = i_3) \dots]
$$

Similarly we can take any relation $R_2$ of $S$ which is not $\true$-valid and is different from $[\neg x]$
or which is $\true$-valid but not $\true$-default
to construct an $S$-formula $\psi_2 \in Gen^{*}_{NC}(S)$ with $y \not \in \opvars{\psi_2}$ 
and constants $i_2, i_3, \dots$
such that $[R_2(x_1, \dots) \wedge \psi_2] = [\neg x_1 \wedge (x_2 = i_2) \wedge (x_3 = i_3) \dots]$.
This finishes the proof.
\end{proof}

%% file: parts/ramsey-schaefer-long.tex

The proof of Theorem~\ref{thm:dichotomy-ramsey-schaefer}
can be split into four steps, each of them being dichotomies themselves.
The first one, Theorem~\ref{thm:schaefer-ramsey-rca-rcolor}, states
the existence of a gap between provability in $\rca$
and implying $\lrsats{[x \neq y]}$ over $\rca$.
Then we focus successively on two classes of boolean formulas:
bijunctive formulas (Theorem~\ref{thm:schaefer-bij-dichotomy})
and affine formulas (Theorem~\ref{thm:schaefer-affine-dichotomy})
whose corresponding principles happen to be either a consequence of
$\lrsats{[x \neq y]}$ or equivalent to the full class of bijunctive (resp. affine) formulas.
Remaining cases are handled by Theorem~\ref{thm:dichotomy-remaining}.
We first state a trivial relation between a satisfaction principle and its Ramseyan version.

\begin{lemma}\label{lem:sats-lrsats}
$\rca \vdash \sats{S} \imp \lrsats{S}$
\end{lemma}
\begin{proof}
Let $(V, C, L)$ be an instance of $\lrsats{S}$. Let $\nu: V \to \Bool$ be a satisfying
assignment for $C$. Then either $\set{x \in L : \nu(x) = \true}$ or $\set{x \in L : \nu(x) = \false}$
is infinite, and both sets exist by $\Delta^0_1$-comprehension.
\end{proof}

\begin{definition}
Let $S$ be a set of relations over Booleans. 
The class of existentially quantified $S$-formulas with constants and closed under conjunction
-- i.e. of the form $(\exists \vec{x})\bigwedge_{i < n} R_i[\vec{x}, \vec{y}, \true, \false]$ with $R_i \in S$ -- is denoted
by $Gen(S)$. We also define $Rep(S) = \set{[R] : R \in Gen(S)}$, ie. the
relations represented by existentially quantified $S$-formula with constants and closed under conjunction.
By abuse of notation, we may use $Rep(R)$ when $R$ is a relation to denote $Rep(\set{R})$.
We can also define similar relations without constants, denoted by $Gen_{NC}$ and $Rep_{NC}$.
\end{definition}

\begin{lemma}\label{lem:rep-lrsats}
$\rca$ proves: If $T$ is a sequence of Boolean relations such that $[x \neq y] \in Rep_{NC}(T)$,
and~$S$ is a sequence of relations in~$Rep_{NC}(T)$, 
then $\lrsats{T} \imp \lrsats{S}$.
\end{lemma}
\begin{proof}
Let $(V, C, L)$ be an instance of $\lrsats{S}$. Say $V = \set{x_0, x_1, \dots}$ and $C = \set{\varphi_0, \varphi_1, \dots}$.
Define an instance $(V \cup F, D, L)$ of $\lrsats{T}$
where $F = \set{y_0, y_1, \dots}$ is a set of fresh variables disjoint from~$V$,
and $D$ is a set of formulas defined by stages as follows.
At stage 0, $D = \emptyset$.
In order to make $D$ computable, we will ensure that after stage $s$,
no formula over $\set{x_i, y_i : i < s}$ will be added to~$D$.
At stage $s$, we want to add constraints of $\varphi_s$ to~$D$.
Because $S \subseteq Rep_{NC}(T)$ and $T$ is c.e., we can effectively find a formula 
$\psi \in Gen_{NC}(T)$ equivalent to~$\varphi_s$ and translate it into a finite set of formulas $\psi^{*}$ as follows:
$(\exists z.\psi_1)^{*} \simeq (\psi_1[y/z])^{*}$ where $y \in F$ is a fresh variable,
$(\psi_1 \wedge \psi_2)^{*} \simeq \psi_1^{*} \cup \psi_2^{*}$, 
$R(x_{i_1}, \dots, x_{i_n})^{*} \simeq \{R(y_{j_1}, \dots, y_{j_n}), x_{i_1} = y_{j_1}, \dots, x_{i_n} = y_{j_n} \}$
where $y_{j_k}$ are fresh variables of $F$ and $x = y$ is a notation for the composition of $(\exists z)x \neq z \wedge z \neq y$.
Add $\psi^{*}$ to~$D$. It is easy to check that any solution to 
$(V \cup F \cup \set{c_0, c_1}, D, L)$ is a solution to $(V, C, L)$.
\end{proof}

\subsection{From provability to~$\lrsats{[x \neq y]}$}

Our first dichotomy for Ramseyan principles is between $\rca$ and $\lrsats{[x \neq y]}$.

\begin{theorem}\label{thm:schaefer-ramsey-rca-rcolor}
If $S$ satisfies one of the conditions (a)-(d) below then $\rca \vdash \lrsats{S}$.
Otherwise $\rca \vdash \lrsats{S} \imp \lrsats{[x \neq y]}$.
\begin{itemize}
\begin{minipage}[t]{0.5\linewidth} 
  \item[(a)] Every relation in $S$ is $\false$-valid.
  \item[(b)] Every relation in $S$ is $\true$-valid.
\end{minipage}
\begin{minipage}[t]{0.5\linewidth} 
  \item[(c)] Every relation in $S$ is horn.
  \item[(d)] Every relation in $S$ is co-horn.
\end{minipage}
\end{itemize}
\end{theorem}

The proof of Theorem~\ref{thm:schaefer-ramsey-rca-rcolor} follows Theorem~\ref{thm:rca-proves-lrsathorn}.

\begin{lemma}[{{Schaefer in \cite[3.2.1]{schaefer1978complexity}}}]\label{lem:schaefer-3.2.1}
$\rca$ proves: If $S$ contains some relation which is not horn and some relation which
is not co-horn, then $[x \neq y] \in Rep(S)$.
\end{lemma}

\begin{lemma}\label{lem:case-neq-or-simple}
$\rca$ proves that at least one of the following holds:
\begin{itemize}
\begin{minipage}[t]{0.5\linewidth} 
  \item[(a)] Every relation in $S$ is $\false$-valid.
  \item[(b)] Every relation in $S$ is $\true$-valid.
  \item[(c)] Every relation in $S$ is horn.
\end{minipage}
\begin{minipage}[t]{0.5\linewidth} 
  \item[(d)] Every relation in $S$ is co-horn.
  \item[(e)] $[x \neq y] \in Rep_{NC}(S)$.
\end{minipage}
\end{itemize}
\end{lemma}
\begin{proof}
Assume none of cases (a), (b) and (e) holds.
Then by Lemma~\ref{lem:schaefer-4.3}, $[x]$ and $[\neg x]$ are contained in $Rep_{NC}(S)$,
hence $Rep_{NC}(S) = Rep(S)$. So by Lemma~\ref{lem:schaefer-3.2.1}, either every relation in $S$
is horn, or every relation in $S$ is co-horn. 
\end{proof}

It is easy to see that $\lrsatzerovalid$ and $\lrsatonevalid$ both hold over $\rca$.
We will now prove that so do $\lrsathorn$ and $\lrsatcohorn$, but first we must introduce the powerful tool
of \emph{closure under functions}.

\begin{definition}
We say that a relation $R \subseteq \Bool^n$ is \emph{closed} or \emph{invariant} under an $m$-ary function $f$
and that $f$ is a \emph{polymorphism} of $R$ if
for every $m$-tuple $\tuple{v_1, \dots, v_m}$ of vectors of $R$, $\vec{f}(v_1, \dots, v_m) \in R$
where $\vec{f}$ is the coordinate-wise application of the function $f$.
\end{definition}

We denote the set of all polymorphisms of $R$ by $\poly{R}$, and for a set $\Gamma$
of Boolean relations we define $\poly{\Gamma} = \set{ f : f \in \poly{R} \mbox{ for every } R \in \Gamma}$.
Similarly for a set $B$ of Boolean functions, $\inv{B} = \set{R : B \subseteq \poly{R}}$
is the set of \emph{invariants} of $B$.
One easily sees that the projection functions are polymorphism of every Boolean relation~$R$.
In particular, the identity function is a polymorphism of~$R$.
As well, the composition of polymorphisms of~$R$ form again a polymorphism of~$R$.
So given a set of Boolean relations~$S$, $\poly{S}$ contains all projection functions
and is closed under composition. The sets of functions satisfying those closure properties
have been studied in universal algebra under the name of~\emph{clones}.
We have seen that for every set of Boolean relations~$S$, $\poly{S}$ is a clone.
Post~\cite{post1942two} studied the lattice of clones of Boolean functions and proved that
they admit a finite basis.

The lattice structure of the Boolean clones has connections with the complexity
of satifiability problems. Indeed, if some clone~$A$ is a subset of another clone~$B$,
then~$\inv{A} \supseteq \inv{B}$. But then trivially~$\lrsats{\inv{A}} \imp \lrsats{\inv{B}}$.
As well, we shall see that as soon as~$[x = y] \in Rep_{NC}(S)$, the sets $Rep_{NC}(S)$ and $\inv{\poly{S}}$
coincide. Therefore, assuming that the equality relation is representable in~$S$, 
the study of the strength of~$\lrsats{S}$ can be reduced to the study of
the strength of~$\lrsats{\inv{A}}$ for every clone in Post's lattice.

\begin{definition}
The \textit{conjunction function} $\mathsf{conj} : \Bool^2 \to \Bool$ is defined by
$\mathsf{conj}(a, b) = a \wedge b$, the \textit{disjunction function} $\mathsf{disj} : \Bool^2 \to \Bool$ 
by $\mathsf{disj}(a, b) = a \vee b$, the \textit{affine function} $\mathsf{aff} : \Bool^3 \to \Bool$ by
$\mathsf{aff}(a, b, c) = a \oplus b \oplus c = \true$ and the \textit{majority function} 
$\mathsf{maj} : \Bool^3 \to \Bool$ by
$\mathsf{maj}(a, b, c) = (a \wedge b) \vee (a \wedge c) \vee (b \wedge c)$.
\end{definition}

The following theorem due to Schaefer characterizes
relations in terms of closure under some functions.
The proof is relativizable and involves finite objects. Hence
it can be easily proven to hold over $\rca$.

\begin{theorem}[Schaefer \cite{schaefer1978complexity}]\label{thm:schaefer-closure}
A relation is
\begin{enumerate}
  \item horn iff it is closed under conjunction function
  \item co-horn iff it is closed under disjunction function
  \item affine iff it is closed under affine function
  \item bijunctive iff it is closed under majority function
\end{enumerate}
\end{theorem}

In other words, using Post's lattice, a relation $R$ is horn iff $E_2 \subseteq \poly{R}$,
co-horn iff $V_2 \subseteq \poly{R}$, affine iff $L_2 \subseteq \poly{R}$ and bijunctive
iff $D_2 \subseteq \poly{R}$.
In the case of horn and co-horn relations, we will use the closure of the valid assignments
under the conjunction and disjunction functions to prove that~$\lrsathorn$
and~$\lrsatcohorn$ both hold over~$\rca$.
%

\begin{theorem}\label{thm:rca-proves-lrsathorn}
If every relation in $S$ is horn (resp. co-horn) then $\rca \vdash \lrsats{S}$.
\end{theorem}
\begin{proof}
We will prove it over $\rca$ for the horn case. The proof for co-horn relations is similar.
Let $(V, C, L)$ be an instance of $\lrsathorn$
and $F \subseteq L$ be the collection of variables $x \in L$ such that there exists a finite set $C_{fin} \subseteq C$ 
for which every valid assignment $\nu$ satisfies $\nu(x) = \true$.

Case 1: $F$ is infinite.
Because $F$ is $\Sigma^0_1$, we can
find an infinite $\Delta^0_1$ subset $H$ of $F$.
The set $H$ is homogeneous for~$C$ with color~$\true$.
  
Case 2: $F$ is finite.
We take $H = L \setminus F$ and claim that $H$ is homogeneous for~$C$ with color~$\false$.
If $H$ is not homogeneous for $C$, then there exists a finite set
$C_{fin} \subseteq C$ witnessing it. Let $H_{fin}  = \opvars{C_{fin}} \cap H$.
By definition of not being homogeneous with color~$\false$, for every assignment $\nu$ satisfying $C_{fin}$,
there exists a variable $x \in H_{fin}$ such that $\nu(x) = \true$.
By definition of~$H$, for each variable $x \in H$ there exists a valid assignment $\nu_x$ such that
$\nu_x(x) = \false$. By Theorem~\ref{thm:schaefer-closure}, the class valid assignments 
of a finite horn formula is closed under conjunction.
So $\nu = \bigwedge_{x \in H_{fin}} \nu_x$ is a valid assignment for $C_{fin}$ such that
$\nu(x) = \false$ for each $x \in H_{fin}$. Contradiction. 
\end{proof}

\begin{proof}[Proof of Theorem~\ref{thm:schaefer-ramsey-rca-rcolor}]
If every relation in $S$ is $\false$-valid (resp. $\true$-valid) then $\lrsats{S}$
holds obviously over $\rca$. If every relation in $S$ is horn (resp. co-horn)
then by Theorem~\ref{thm:rca-proves-lrsathorn}, $\lrsats{S}$ holds also over $\rca$.
By Lemma~\ref{lem:case-neq-or-simple}, the only remaining case is where $[x \neq y] \in Rep_{NC}(S)$.
There exists a finite (hence c.e.) subset $T \subseteq S$ such that $[x \neq y] \in Rep_{NC}(T)$.
By Lemma~\ref{lem:rep-lrsats}, $\rca \vdash \lrsats{T} \imp \lrsats{[x \neq y]}$,
hence $\rca \vdash \lrsats{S} \imp \lrsats{[x \neq y]}$. 
\end{proof}

The following technical lemma will be very useful for the remainder of the paper.

\begin{lemma}\label{lem:fix-color-lrsats}
$\rca$ proves the following: Suppose $T$ is a sequence of Boolean relations such that 
\begin{itemize}
	\item[1.] $T$ contains a relation which is not $\false$-valid
	\item[2.] $T$ contains a relation which is not $\true$-valid
	\item[3.] $[x \neq y] \in Rep_{NC}(T)$
\end{itemize}
If $S$ is a sequence such that $S \subseteq Rep_{NC}(T \cup \set{[x], [\neg x]})$
then $\lrsats{T} \imp \lrsats{S}$.
\end{lemma}
\begin{proof}
We reason by case analysis.
Suppose that $[x]$ and $[\neg x]$ are both in $Rep_{NC}(T)$.
Then $S \subseteq Rep_{NC}(T)$, so by Lemma~\ref{lem:rep-lrsats}, $\rca \vdash \lrsats{T} \imp \lrsats{S}$.

Suppose now that either $[x]$ or $[\neg x]$ is not in $Rep_{NC}(T)$.
Then by Lemma~4.3 of~\cite{schaefer1978complexity}, every relation in $T$ is complementive,
that is, if $\vec{r} \in R$ for some $R \in T$, then the pointwise negation of $\vec{r}$ is also in~$R$.
By Lemma~\ref{lem:rep-lrsats}, it suffices to ensure that $\rca \vdash \lrsats{Rep_{NC}(T)} \imp \lrsats{T \cup \set{[x], [\neg x]}}$
to conclude, as $\rca \vdash \lrsats{T} \imp \lrsats{Rep_{NC}(T)}$.
Let $(V, C, L)$ be an instance of $\lrsats{T \cup \set{[x], [\neg x]}}$. 
Say $V = \set{x_0, x_1, \dots}$ and $C = \set{\varphi_0, \varphi_1, \dots}$.
Define an instance $(V \cup \set{c_0, c_1}, D, L)$ of $\lrsats{Rep_{NC}(T)}$ such that~$c_0, c_1 \not \in V$
and with the set of formulas
$$
D = \{ c_0 \neq c_1 \} \cup \{ R(\vec{x}) \in C : R \neq [x] \wedge R \neq [\neg x]\}
\cup \{ x = c_0 : (\neg x) \in C \} \cup \{ x = c_1 : (x) \in C\}
$$
Note that $[x = y] \in Rep_{NC}(T)$ as $[x = y] = [(\exists z)x \neq z \wedge z \neq y]$ and $[x \neq y] \in Rep_{NC}(T)$.
The instance $(V \cup \set{c_0, c_1}, D, L)$ is obviously finitely satisfiable
as every valid assignment $\nu$ of $(V, C, L)$ induces an assignment of $(V \cup \set{c_0, c_1}, D, L)$
by setting $\nu(c_0) = \false$ and $\nu(c_1) = \true$.
Conversely, we prove that for every assignment $\nu$ satisfying $(V \cup \set{c_0, c_1}, D, L)$,
the assignment $\mu$ defined to be $\nu$ if $\nu(c_0) = \false$ and the pointwise negation of $\nu$
if $\nu(c_0) = \true$ satisfies $(V, C, L)$.
Suppose there exists a finite subset $E \subset C$ such that $\mu(\bigwedge E) = \false$.
For every formula $(\neg x) \in E$, $\mu(x) = \mu(c_0) = \false$ and for every $(x) \in E$, $\mu(x) = \mu(c_1) = \true$.
So there must exist a relation $R \in T$ such that $R(\vec{x}) \in E$ and $\mu(R(\vec{x})) = \false$.
By complementation of $R$, $\nu(R(\vec{x})) = \false$, but $R(\vec{x}) \in D$, contradicting the assumption
that $\nu$ satisfies~$D$.
Therefore, every infinite set~$H \subseteq L$ homogeneous for~$D$ is homogeneous for~$C$.
\end{proof}

\subsection{Bijunctive satisfiability}

Our second dichotomy theorem concerns bijunctive relations.
Either the related principle is a consequence of $\lrsats{[x \neq y]}$
over $\rca$, or it has full strength of $\lrsatbijunctive$.
In the remainder of this subsection, we will make the following assumptions
and denote them by the shorthand in the right column of the table:
\begin{tabular}{rll}
	(i) & $S$ contains only bijunctive relations  & ($D_2 \subseteq \poly{S}$)\\
	(ii) & $S$ contains a relation which is not $\false$-valid  & ($I_0 \not \subseteq \poly{S}$)\\
	(iii) & $S$ contains a relation which is not $\true$-valid & ($I_1 \not \subseteq \poly{S}$)\\
	(iv) & $[x \neq y] \in Rep_{NC}(S)$ & ($\poly{S} \subseteq D$)\\
\end{tabular}

\begin{theorem}\label{thm:schaefer-bij-dichotomy}
If $S$ contains only affine relations then $\rca \vdash \lrsats{[x \neq y]} \imp \lrsats{S}$.
Otherwise $\rca \vdash \lrsats{S} \biimp \lrsatbijunctive$.
\end{theorem}

The proof of Theorem~\ref{thm:schaefer-bij-dichotomy} follows Lemma~\ref{lem:bijunctive-cases}.

\begin{definition}
For any set $S$ of relations, the \emph{co-clone} of $S$ is the closure
of $S$ by existential quantification, equality and conjunction. We denote it by $\coclone{S}$.
\end{definition}

Remark that in general, $Rep_{NC}(S)$ may be different from $\coclone{S}$
if $[x = y] \not \in Rep_{NC}(S)$. However in our case, we assume that $[x \neq y] \in Rep_{NC}(S)$,
hence $[x = y] \in Rep_{NC}(S)$ and $Rep_{NC}(S) = \coclone{S}$.
The following property will happen to be very useful for proving that a relation $R \in Rep_{NC}(S)$.

\begin{lemma}[Folklore]
$\inv{\poly{S}} = \coclone{S}$
\end{lemma}

\begin{lemma}\label{lem:bijunctive-cases}
One of the following holds:
\begin{itemize}
  \item[(a)] $Rep_{NC}(S)$ contains all bijunctive relations. 
  \item[(b)] $S \subseteq Rep_{NC}(\set{[x], [x \neq y]})$.
\end{itemize}
\end{lemma}
\begin{proof}
By the blanket assumption of the subsection, $D_2 \subseteq \poly{S} \subseteq D$. 
Either $D_1 \subseteq \poly{S}$ or $\poly{S} = D_2$.
If $D_1 \subseteq \poly{S}$, then every relation in $S$ is affine, so $S \subseteq \inv{D_1} = Rep_{NC}(\set{[x], [x \neq y]})$. 
If $\poly{S} = D_2$ then $Rep_{NC}(S) = \coclone{S} = \inv{\poly{S}} = \inv{D_2}$ 
which is the set of all bijunctive relations.
\end{proof}

\begin{proof}[Proof of Theorem~\ref{thm:schaefer-bij-dichotomy}]
By Lemma~\ref{lem:bijunctive-cases}, either $Rep_{NC}(S)$ contains all bijunctive relations
or $S \subseteq Rep_{NC}(\set{[x], [x \neq y]})$.
In the latter case, by Lemma~\ref{lem:fix-color-lrsats} $\lrsats{[x \neq y]}$ implies $\lrsats{S}$
over $\rca$. In the former case, there exists a finite basis $S_0 \subseteq S$
such that $Rep_{NC}(S_0)$ contains all bijunctive relations.
In particular $S_0$ is a c.e. set, so $\rca \vdash \lrsats{S_0} \imp \lrsatbijunctive$.
Any instance of $\lrsats{S_0}$ being an instance of $\lrsats{S}$,
$\rca \vdash \lrsats{S} \imp \lrsatbijunctive$. The reverse implication follows directly
from the assumption that every relation in $S$ is bijunctive.
So $\rca \vdash \lrsats{S} \biimp \lrsatbijunctive$. 
\end{proof}

\subsection{Affine satisfiability}

In this section, we will prove that if $S$ satisfies none of the previous cases
and contains only affine relations, then the corresponding Ramseyan satisfaction problem
is equivalent to $\lrsataffine$ over $\rca$. So suppose that
\begin{itemize}
	\item[(i)] $S$ contains only affine relations \hfill ($L_2 \subseteq \poly{S}$)
	\item[(ii)] $S$ contains a relation which is not bijunctive \hfill ($D_2 \not \subseteq \poly{S}$)
	\item[(iii)] $S$ contains a relation which is not $\false$-valid \hfill ($I_0 \not \subseteq \poly{S}$)
	\item[(iv)] $S$ contains a relation which is not $\true$-valid \hfill ($I_1 \not \subseteq \poly{S}$)
	\item[(v)] $[x \neq y] \in Rep_{NC}(S)$ \hfill ($\poly{S} \subseteq D$)
\end{itemize}
In particular, $\poly{S} \subsetneq D$.

\begin{theorem}\label{thm:schaefer-affine-dichotomy}
$\rca \vdash \lrsats{S} \biimp \lrsataffine$
\end{theorem}
\begin{proof}
By assumption, every relation in $S$ is affine. Hence $\rca \vdash \lrsataffine \imp \lrsats{S}$.
As $L_2 \subseteq \poly{S} \subsetneq D$, $\poly{S}$ is either $L_3$ or $L_2$.
In particular, $\s{Pol}(S \cup \{[x], [\neg x]\}) = L_2$ 
Considering the corresponding invariants, 
$$
\inv{L_2} \subseteq \s{Inv}(\s{Pol}(S \cup \set{[x], [\neg x]})) = 
\coclone{S \cup \set{[x], [\neg x]}} = Rep_{NC}(S \cup \set{[x], [\neg x]})
$$
There exists a finite basis $S_0$ such that $Rep_{NC}(S_0)$ contains all affine relations.
$\inv{L_2}$ being the set of affine relations, $S_0 \subset Rep_{NC}(S \cup \set{[x], [\neg x]})$.
There exists a finite (hence c.e.) subset $T$ of $S$ such that $S_0 \subseteq Rep_{NC}(T \cup \set{[x], [\neg x]})$.
In particular,  
$$
\{R : R \mbox{ is affine}\} \subseteq Rep_{NC}(S_0) \subseteq Rep_{NC}(T \cup \set{[x], [\neg x]})
$$
By Lemma~\ref{lem:fix-color-lrsats}, $\rca \vdash \lrsats{T} \imp \lrsataffine$, hence
 $\rca \vdash \lrsats{S} \imp \lrsataffine$. 
\end{proof}

\subsection{Remaining cases}

Based on Post's lattice, 
the only remaining cases are $\poly{S} = N_2$ or $\poly{S} = I_2$.

\begin{theorem}\label{thm:dichotomy-remaining}
If $\poly{S} \subseteq N_2$ then $\rca \vdash \lrsats{S} \biimp \lrsat$.
\end{theorem}
\begin{proof}
The direction $\rca \vdash \lrsat \imp \lrsats{S}$ is obvious. We will prove the converse. 
Because $\poly{S} \subseteq N_2$, $\poly{S \cup \set{[x]}} = I_2$.
$$
Rep_{NC}(S \cup \set{[x]}) = \coclone{S \cup \set{[x]}} = \inv{\poly{S \cup \set{[x]}}} \supseteq \inv{I_2}
$$
Note that $\inv{I_2}$ is the set of all Boolean relations. As $\inv{I_2}$ has a finite basis,
there exists a finite $S_0 \subseteq S$ such that $Rep_{NC}(S_0 \cup \set{[x]})$ contains all Boolean relations.
By Lemma~\ref{lem:fix-color-lrsats}, $\rca \vdash \lrsats{S_0} \imp \lrsat$.
Hence $\rca \vdash \lrsats{S} \biimp \lrsat$. 
\end{proof}

\begin{proof}[Proof of Theorem~\ref{thm:dichotomy-ramsey-schaefer}]
By case analysis over $\poly{S}$.
If $I_0$, $I_1$, $E_2$ and $V_2$ are included in $\poly{S}$ (that is, if $S$ contains
only $\false$-valid, $\true$-valid, horn or co-horn relations) then
by Theorem~\ref{thm:schaefer-ramsey-rca-rcolor}, $\rca \vdash \lrsats{S}$.
If $D_1 \subseteq \poly{S} \subseteq D$ then $\rca \vdash \lrsats{S} \biimp \lrsats{[x \neq y]}$ 
by Theorem~\ref{thm:schaefer-bij-dichotomy}.
By the same theorem, if $\poly{S} = D_2$ then $\rca \vdash \lrsats{S} \biimp \lrsatbijunctive$.
If $L_2 \subseteq \poly{S} \subseteq L_3$ then by Theorem~\ref{thm:schaefer-affine-dichotomy},
$\rca \vdash \lrsats{S} \biimp \lrsataffine$. Otherwise, $I_2 \subseteq \poly{S} \subseteq N_2$
in which case $\rca \vdash \lrsats{S} \biimp \lrsat$ by Theorem~\ref{thm:dichotomy-remaining}.
\end{proof}

The principle $\lrsats{[x \neq y]}$ coincides with an already existing principle
about bipartite graphs.
For $k \in \Nb$, we say that a graph $G = (V,E)$ is \emph{$k$-colorable} 
if there is a function $f \colon V \imp k$ such that $(\forall (x,y) \in E)(f(x) \neq f(y))$, 
and we say that a graph is \emph{finitely $k$-colorable} if every finite induced subgraph is $k$-colorable.

\begin{definition}
Let $G = (V,E)$ be a graph.  A set $H \subseteq V$ is \emph{homogeneous for $G$} 
if every finite $V_0 \subseteq V$ induces a subgraph that is $k$-colorable 
by a coloring that colors every $v \in V_0 \cap H$ color $0$.
$\lrcolor_k$ is the statement ``For every infinite, finitely 
$k$-colorable graph $G = (V,E)$ and every infinite $L \subseteq V$ there 
exists an infinite $H \subseteq L$ that is homogeneous for~$G$''. $\rcolor_k$ 
is the restriction of $\lrcolor_k$ with $L = V$.
An instance of $\lrcolor_k$ is a pair $(G, L)$. For $\rcolor_k$, 
it is simply the graph~$G$.
\end{definition}

\begin{theorem}
$\rca \vdash \rcolor_2 \biimp \lrsats{[x \neq y]}$
\end{theorem}
\begin{proof}
See~\cite{bienvenurwkl} for a proof of $\rca \vdash \rcolor_2 \biimp \lrcolor_2$.
There exists a direct mapping between an instance $(V, C, L)$ of $\lrsats{[x \neq y]}$
and an instance $(G, L)$ of $\lrcolor_2$ where $G = (V, E)$ by taking
$E = \{ \set{x, y} : x \neq y \in C \}$. 
\end{proof}

%% file: parts/variants-long.tex
Localized principles are relatively easy to manipulate
as they can express relations defined using existential quantifier
by restricting the localized set $L$ to the variables not captured by any quantifier.
However we will see that when the set of relations has some good closure
properties, the unlocalized version of the principle is as expressive as its localized one.

\begin{theorem}\label{thm:localized-to-unlocalized}
$\rca$ proves that if $S$ be a $\Sigma^0_1$ \emph{co-clone} then $\rsats{S} \biimp \lrsats{S}$
\end{theorem}
\begin{proof}
The implication $\lrsats{S} \imp \rsats{S}$ is obvious. 
To prove the converse, let $(V, C, L)$ be an instance of $\lrsats{S}$
with $V = \set{x_i : i \in \omega}$ and $C = \set{\theta_i : i \in \omega}$.
Let $C_L$ be a computable enumeration of formulas $\phi(\vec{x}) = R(\vec{x})$ with $R \in S$ and $\vec{x} \subset L$ 
such that there exists a finite subset $C_{fin}$ of $C$ for which
every valid truth assignment $\nu$ over $C_{fin}$ satisfies $\nu(\phi) = \true$.

If $C_L$ is finite, then there is a bound $m$ such that if $\phi \in C_L$
then $max(i : x_i \in \opvars{\phi}) \leq m$. Then take $H = \set{x_i \in L : i > m}$.
$H \subseteq L$ and is infinite because~$L$ is infinite.

\begin{claim}
For every $c \in \Bool$, $H$ is homogeneous for $C$ with color~$c$.
\end{claim}
\emph{Proof of claim.}
If not then there exists a finite subset $C_{fin}$ of $C$
such that $H$ is not homogeneous for $C_{fin}$ with color~$c$.
Let $\vec{y} = \opvars{C_{fin}} \setminus L$.
Because $S$ is a co-clone, it is closed under finite conjunction and projection, hence
$(\exists \vec{y})\bigwedge C_{fin}$ is equivalent to an $S$-formula, say $\varphi$.
In particular $\opvars{\varphi} \subseteq \opvars{C_{fin}} \cap L$ and $\varphi \in C_L$.
For every assignment $\nu$ satisfying $\varphi$, there is a variable $x \in H$ such that $\nu(x) = \neg c$.
Then $\opvars{\varphi} \cap H \neq \emptyset$. However $\varphi \in C_L$, so 
$\opvars{\varphi} \cap H = \emptyset$ by definition of $H$. Contradiction. This finishes
the proof of the claim.

So suppose instead $C_L = \set{\phi_i : i \in \Nb}$ is infinite, and suppose each $\phi_i$ is unique.
We construct an instance $(V', C')$ of $\rsats{S}$ by taking
$V' = L \cup \set{y_n : n \in \Nb}$ and constructing $C'$ by stages as follows:
At stage 0, $C' = \emptyset$.
At stage $s+1$, look at $\phi_s = R(x_1, \dots, x_m)$ and let $x_i$ be the greatest variable
in lexicographic order among $x_1, \dots, x_m$. Add the formula $x_i = y_s$
and the formula $R(x_1, \dots, x_{i-1}, y_s, x_{i+1}, \dots, x_m)$ to $C'$.
Then go to next stage.
This finishes the construction.
Note that $C'$ is satisfiable, otherwise
by definition there would be a finite unsatisfiable subset $C_{fin} \subset C_L$
from which we could extract an unsatisfiable subset of $C$.
Also note that, by assuming that $\phi_i$ is unique and $x_i$ is the greatest variable
in lexicographic order, the number of stages $s$ such that the formula $x = y_s$
is added to $C'$ is finite for each variable~$x$.

Let $H'$ be an infinite set homogeneous for $C'$ with color $c$.
We can extract from $H'$ an infinite subset of $L$ homogeneous for $C'$ with color $c$
because either $L \cap H'$ or $\set{x \in L: (x = y_n) \in C' \mbox{ and } y_n \in H'}$
is infinite and both are homogeneous for $C'$ with color $c$.
So fix $H \subseteq L$, an infinite set homogeneous for $C'$ (and for $C_L$) with color~$c$.

\begin{claim}
$H$ is homogeneous for $C$ with color $c$.
\end{claim}
\emph{Proof of claim.}
By the same argument as previous claim, suppose there is a finite subset $C_{fin}$ of~$C$
such that $H$ is not homogeneous for $C_{fin}$ with color~$c$. 
Let $\varphi$ be the $S$-formula equivalent to~$(\exists \vec{y})\bigwedge C_{fin}$
where $\vec{y} = \opvars{C_{fin}} \setminus L$.
For every valid assignment $\nu$ for $\varphi$, there is a variable $x \in H$ such that $\nu(x) = \neg c$.
But $\varphi \in C_L$ and hence $H$ is homogeneous for $\varphi$ with color~$c$. Contradiction. 
This last claims finishes the proof of Theorem~\ref{thm:localized-to-unlocalized}.
\end{proof}

Noticing that affine (resp. bijunctive) relations form a co-clone, we immediately 
deduce the following corollary.

\begin{corollary}\label{coro:localized-affine-bijunctive}
$\rsataffine$ and $\rsatbijunctive$ are equivalent to their local version
over $\rca$.
\end{corollary}

A useful principle below $\wkl$ for studying the strength of a statement
is the notion of \emph{diagonally non-computable function}.

\begin{definition}
A total function $f$ is \emph{diagonally non-computable} if $(\forall e)f(e) \neq \Phi_e(e)$.
$\dnr$ is the corresponding principle, i.e. for every $X$, there exists a function d.n.c. relative to~$X$.
\end{definition}

$\dnr$ is known to coincide with the restriction of $\rwkl$ to trees of positive measure
(\cite{flood2012reverse,bienvenurwkl}). On the other side, there exists an $\omega$-model
of $\dnr$ which is not a model of $\rcolor_2$ (\cite{bienvenurwkl}).
We will now prove that we can compute a diagonally non-computable function
from any infinite set homogeneous for a particular set
of affine formulas. As $\rsat$ implies $\lrsataffine$ over $\rca$,
it gives another proof of $\rca \vdash \rwkl \imp \dnr$.

\begin{theorem}\label{RSATfin3provesDNR}
$\rca \vdash \rsataffine \imp \dnr$.
\end{theorem}
\begin{proof}
We construct a computable set $C$ of affine formulas over a computable set $V$ of variables
such that every infinite set homogeneous for $C$ computes a diagonally non-computable function. Relativization is straightforward.
Let $t : \Nb \to \Nb$ be the computable function defined by
$t(0) = 2$ and $t(e+1) = 2 + \sum_{i=0}^e t(i)$.
Note that every image by $t$ is even.
For every $e \in \Nb$, let $\tuple{D_{e,j} : j \in \Nb}$ denote the canonical enumeration
of all finite sets of size $t(e)$.
We fix a countable set of variables $V = \set{x_0, x_1, \dots}$ a define a set of formulas $C$ satisfying the following requirements:
$$
  \Rcal_e:\ \Phi_e(e) \downarrow \Rightarrow D_{e, \Phi_e(e)} \mbox{ is not homogeneous for } C
$$

We first show how to construct a d.n.c. function from an infinite set~$H$ homogeneous for~$C$,
assuming that each requirement is satisfied.
Let $g(\cdot)$ be such that $D_{e,g(e)}$ are the least $t(e)$ elements of $H$.
We claim that $g$ is a d.n.c. function: 
If $\Phi_e(e) \uparrow$ then obviously $g(e) \neq \Phi_e(e)$.
If $\Phi_e(e) \downarrow$ then because of requirement $\Rcal_e$, 
$D_{e,\Phi_e(e)} \cap \bar H \neq \emptyset$, hence $D_{e,g(e)} \neq D_{e, \Phi_e(e)}$
so $g(e) \neq \Phi_e(e)$. 

We define $C$ by stages. At stage 0, $C = \emptyset$. To make $C$ computable, 
we will not add to~$C$ any formula over $\set{x_i : i \leq s}$ after stage $s$.
Suppose at stage $s$ $\Phi_{e,s}(e) \downarrow$ for some $e < s$ -- we can assume w.l.o.g. that at most one $e$
halts at each stage --. Then add $x_s \oplus x_s \bigoplus D_{e,\Phi_{e,s}(e)}$ to~$C$. This finishes stage $s$.
One easily check that each requirements is satisfied as $x_s \oplus x_s \bigoplus D_{e,\Phi_{e,s}}$
is logically equivalent to~$\bigoplus D_{e,\Phi_{e,s}(e)}$, and as $D_{e,\Phi_{e,s}(e)}$ has even size,
so the relation is neither $\false$-valid nor $\true$-valid, hence $D_{e, \Phi_{e,s}(e)}$ is not homogeneous for~$C$.

\begin{claim}
The resulting instance is satisfiable.
\end{claim}
\emph{Proof of claim.}
If not, there exists a finite $C_{fin} \subset C$ which is not satisfiable.
For a given Turing index $e$, define $C_e$ to be the set of formulas 
added in some stage $s$ at which $\Phi_{i,s}(i) \downarrow$ for some $i < e$.
There exists an $e_{max}$ such that $C_{fin} \subseteq C_{e_{max}}$.
We will define a valid assignment $\nu_e$ of $C_e$ by $\Sigma_1$-induction over $e$.

If $e = 0$, then $C_0 = \emptyset$ and $\nu_0 = \emptyset$ is a valid assignment.
Suppose we have a valid assignment $\nu_e$ for some $C_e$. We will construct a valid
assignment $\nu_{e+1}$ for $C_{e+1}$. If $\Phi_{e}(e) \uparrow$ then $C_{e+1} = C_e$ and $\nu_e$ is a valid
assignment for $C_{e+1}$. Otherwise $\Phi_e(e) \downarrow$.
$C_{e+1} = C_e \cup \set{x_s \oplus x_s \bigoplus D_{e,\Phi_e(e)}}$.
$\opvars{C_e}$ has at most $\sum_{i=0}^{e-1}$ elements, hence $D_{e,\Phi_e(e)} \setminus \opvars{C_e}$
is not empty. We can hence easily extend our valuation $\nu_e$ to~$D_{e,\Phi_e(e)}$ such that
the resulting valuation satisfies $C_{e+1}$.
This claim finishes the proof of Theorem~\ref{RSATfin3provesDNR}.
\end{proof}

%% file: parts/conclusion-long.tex
Satisfaction principles happen to collapse in the case of a full assignment existence statement.
The definition is not robust and the conditions of the corresponding dichotomy theorem
evolves if we make the slight modification of allowing conjunctions in our definition of formulas.

However, the proposed Ramseyan version leads to a much more robust dichotomy theorem with four main subsystems.
The conditions of ``tractability'' -- here provability over $\rca$ -- differ from those of Schaefer dichotomy theorem
but the considered classes of relations remain the same.
We obtain the surprising result that infinite versions of horn and co-horn satisfaction problems
are provable over $\rca$ and strictly weaker than bijunctive and affine corresponding principles,
whereas the complexity classification of \cite{allender2005complexity} has shown that
horn satisfiability was P-complete under $AC^0$ reduction, hence at least as strong as
bijunctive satisfiability which is NL-complete.

\subsection{Summary of principles considered}

The following diagram summarizes the known relations
between the principles considered here. Single arrows
express implication over $\rca$. Double arrows mean
that implications are strict. A crossed arrow denotes
a non-implication over $\omega$-models.

\begin{figure}[htbp]
\centering
\includegraphics[width=8cm]{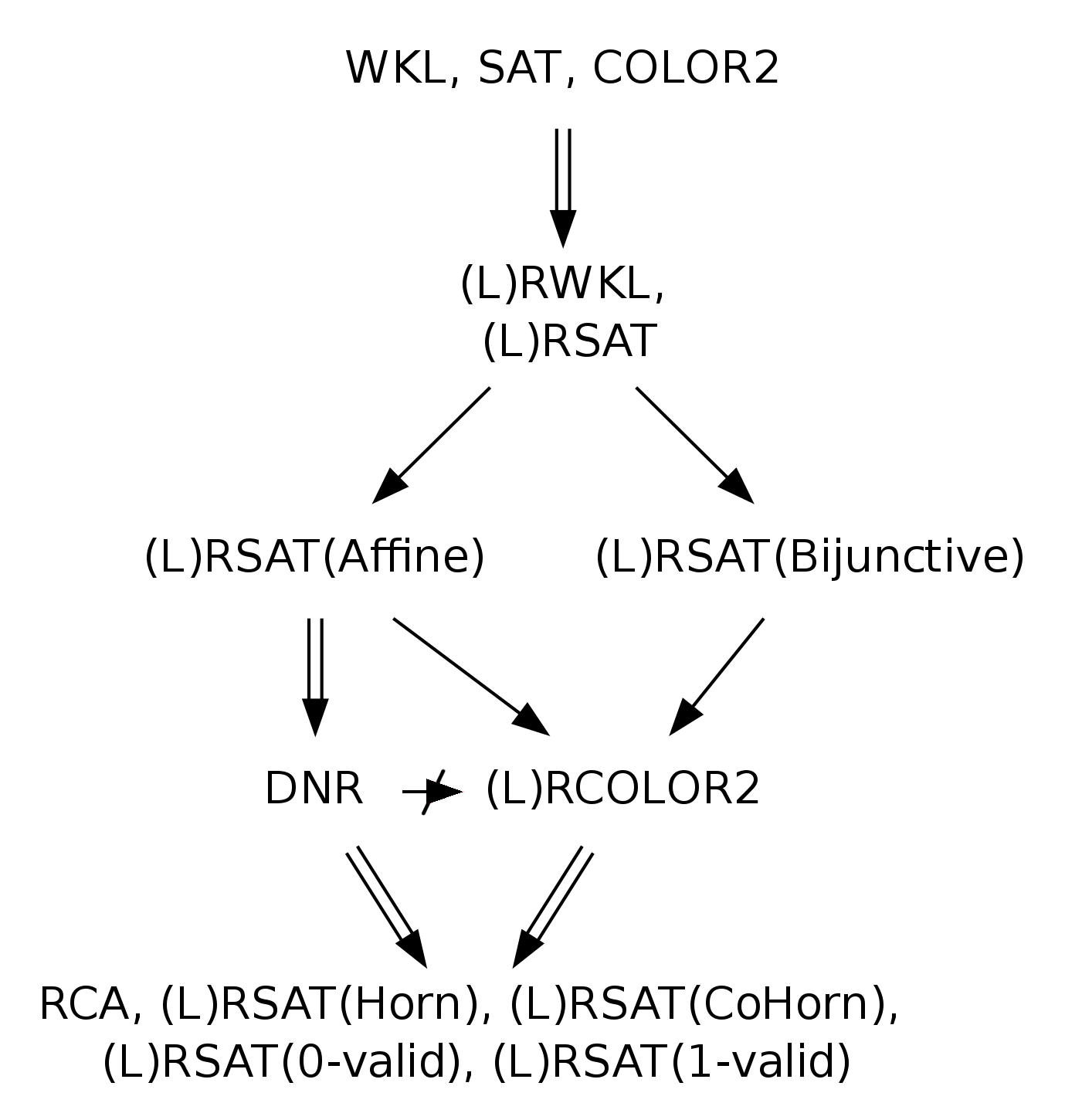}
\caption{Summary of principles}
\end{figure}

Localized and non-localized principles coincide
for the main principles because of
Theorem~\ref{thm:localized-to-unlocalized}.
By \cite{bienvenurwkl}, there exists an $\omega$-model of 
$\dnr$ -- and even $\wwkl$ --which is not a model of $\rcolor_2$.
The missing arrows are all unknown.

\subsection{Open questions}

Very few relations are known between the four main subsystems studied
in this paper: $\rsat$, $\rsataffine$, $\rsatbijunctive$ and $\rcolor_2$.
Theorem~\ref{thm:dichotomy-ramsey-schaefer} states that $\lrsats{S}$
is equivalent to one of the above mentioned principles, or is provable over $\rca$.
However those principles are not known to be pairwise distinct. 
In particular the principle $\rcolor_2$ introduced
in \cite{bienvenurwkl} is not even known to be strictly below $\rwkl$.

\begin{question}
What are the relations between $\rsat$, $\rsataffine$, $\rsatbijunctive$
and $\rcolor_2$~?
\end{question}

\begin{question}
Does $\rcolor_2$ imply $\dnr$ over $\rca$ ? Does it imply $\rwkl$~?
\end{question}

%% file: parts/appendix-long.tex
\section{Post's lattice}

\begin{center}
\includegraphics[scale=0.60]{imgs/clones.mps}
\end{center}

\clearpage

\input{parts/appendix/class02}

\begin{figure}
\begin{center}
\def\hline{\noalign{\hrule height.1pt}}
\def\Hline{\noalign{\hrule height.8pt}}
\begin{tabular}{lll}
\Hline
\textbf{Class} & \textbf{Definition} & \textbf{Base(s)} \\
\Hline
$\BF$ & all Boolean functions & $\{\land, \NOT\}$  \\ \hline
$\Orep$ & $\{\,f \in \BF\mid f$ is 0-reproducing\,\} & $\{\land, \XOR\}$ \\
\hline
$\Nrep$ & $\{\,f \in \BF\mid f$ is 1-reproducing\,\} & $\{\OR, x \oplus y \oplus
1$\} \\ \hline
$\Rep $ & $\Nrep \cap \Orep$ & $\{\OR, x \wedge (y\oplus z \oplus 1)$\} \\
\hline
$\Mon$ & $\{\,f \in \BF\mid f$ is monotonic\,\} & $\{\land, \OR, \czero, \cone$\}
\\ \hline
$\Nmon$ & $\Mon \cap \Nrep$ & $\{\land, \OR, \cone$\} \\ \hline
$\Omon$ & $\Mon \cap \Orep$ & $\{\land, \OR, \czero$\} \\ \hline
$\Tmon$ & $\Mon \cap \Rep$ & $\{\land, \OR\}$ \\ \hline
$\Sep^n_0$ & $\{\,f \in \BF\mid f$ is 0-separating of degree $n$\,\} & $\{\IMP,
\dua(t_n)\}$ \\ \hline
$\Sep_0$ & $\{\,f \in \BF\mid f$ is 0-separating\,\} & $\{\IMP\}$ \\ \hline
$\Sep^n_1$ & $\{\,f \in \BF\mid f$ is 1-separating of degree $n$\,\} & \{$x
\wedge \overline{y}$, $t_n$\} \\ \hline
$\Sep_1$ & $\{\,f \in \BF\mid f$ is 1-separating\,\} & \{$x \wedge
\overline{y}$\} \\ \hline
$\Sep^n_{02}$ & $\Sep^n_0 \cap \Rep$ & \{$x \vee (y \wedge
\overline{z}), \dua(t_n)\}$ \\ \hline
$\Sep_{02}$ & $\Sep_0 \cap \Rep$ & \{$x \vee (y \wedge \overline{z})$\} \\
\hline
$\Sep^n_{01}$ & $\Sep^n_0 \cap \Mon$ & $\{\dua(t_n), \cone\}$ \\ \hline
$\Sep_{01}$ & $\Sep_0 \cap \Mon$ & $\{x \vee (y \wedge z), \cone\}$ \\ \hline
$\Sep^n_{00}$ & $\Sep^n_0 \cap \Rep \cap \Mon$ & $\{x \vee (y \wedge z),
\dua(t_n)\}$ \\ \hline
$\Sep_{00}$ & $\Sep_0 \cap \Rep \cap \Mon$ & $\{x \vee (y \wedge z)\}$ \\ \hline
$\Sep^n_{12}$ & $\Sep^n_1 \cap \Rep$ & $\{x \wedge (y \vee \overline{z}), t_n\}$
\\ \hline
$\Sep_{12}$ & $\Sep_1 \cap \Rep$ & $\{x \wedge (y \vee \overline{z})\}$ \\
\hline
$\Sep^n_{11}$ & $\Sep^n_1 \cap \Mon$ & $\{t_n, \czero\}$ \\ \hline
$\Sep_{11}$ & $\Sep_1 \cap \Mon$ & $\{x \wedge (y \vee z), \czero\}$ \\ \hline
$\Sep^n_{10}$ & $\Sep^n_1 \cap \Rep \cap \Mon$ & $\{x \wedge (y \vee z), t_n\}$
\\ \hline
$\Sep_{10}$ & $\Sep_1 \cap \Rep \cap \Mon$ & $\{x \wedge (y \vee z)\}$ \\ \hline
$\Self$ & $\{\,f\mid f$ is self-dual\,\} & $\{(x\wedge\overline{y}) \vee (x\wedge\overline{z})
\vee (\overline{y}\wedge\overline{z})\}$ \\ \hline
$\Nself$ & $\Self \cap \Rep$ & $\{(x\wedge y) \vee (x\wedge\overline{z}) \vee (y\wedge\overline{z})\}$ \\
\hline
$\Tself$ & $\Self \cap \Mon$ & $\{(x\wedge y) \vee (y\wedge z) \vee (x\wedge z)\}$ \\ \hline
$\Lin$ & $\{\,f\mid$ $f$ is linear\} &
      $\{\XOR, \cone\}$ \\ \hline
$\Olin$ & $\Lin \cap \Orep$ & $\{\XOR\}$ \\ \hline
$\Nlin$ & $\Lin \cap \Nrep$ & $\{\leftrightarrow\}$ \\ \hline
$\Tlin$ & $\Lin \cap \Rep$ & $\{x \oplus y \oplus z\}$ \\ \hline
$\Rlin$ & $\Lin \cap \Self$ & $\{x \oplus y \oplus z \oplus \cone\}$ \\ \hline
$\VelC$ & $\{\,f\mid$ $f$ is an $\OR$-function or a constant function\} &
      $\{\OR, \czero, \cone\}$ \\ \hline
$\Ovel$ & $[\{\OR\}] \cup [\{\czero\}]$ & $\{\OR, \czero\}$ \\ \hline
$\Nvel$ & $[\{\OR\}] \cup [\{\cone\}]$ & $\{\OR, \cone\}$ \\ \hline
$\Tvel$ & $[\{\OR\}]$ & $\{\OR\}$ \\ \hline
$\EtC$ & $\{\,f\mid$ $f$ is an $\land$-function or a constant function\} &
      $\{\land, \czero, \cone\}$ \\ \hline
$\Oet$ & $[\{\land\}] \cup [\{\czero\}]$ & $\{\land, \czero\}$ \\ \hline
$\Net$ & $[\{\land\}] \cup [\{\cone\}]$ & $\{\land, \cone\}$ \\ \hline
$\Tet$ & $[\{\land\}]$ & $\{\land\}$ \\ \hline
$\Neg$ & $[\{\NOT\}] \cup [\{\czero\}] \cup [\{\cone\}]$ & $\{\NOT, \cone\}$,
$\{\NOT, \czero\}$ \\ \hline
$\Tneg$ &  $[\{\NOT\}]$ & $\{\NOT\}$ \\ \hline
$\Ids$ & $[\{\ID\}] \cup [\{\cone\}] \cup [\{\czero\}]$ & $\{\ID, \czero, \cone\}$
\\ \hline
$\OIds$ & $[\{\ID\}] \cup [\{\czero\}]$ & $\{\ID, \czero\}$ \\ \hline
$\NIds$ & $[\{\ID\}] \cup [\{\cone\}]$ & $\{\ID, \cone\}$ \\ \hline
$\TIds$ & $[\{\ID\}]$ & $\{\ID\}$ \\ \Hline
\end{tabular}
\caption{The list of all Boolean clones with definitions and bases, where $t_n
:= \bigvee^{n+1}_{i=1}\bigwedge^{n+1}_{j=1,j\neq i} x_j$ and
      $\dual{f}(a_1, \dots , a_n) = \neg f(\neg a_1 \dots , \neg a_n)$.}
\label{Bases}
\end{center}
\end{figure}

%% file: parts/appendix/class02.tex
\newcommand{\var}[1]{\ensuremath{\text{\upshape{Var}}}(#1)}

\newcommand{\enumsat}[1]{\ensuremath{\mathrm{Enum\ SAT_C}(#1)}}
\newcommand{\FV}[1]{\ensuremath{\mathrm{FV_C}(#1)}}
\newcommand{\USAT}[1]{\ensuremath{\mathrm{Unique\ SAT_C}(#1)}}
\newcommand{\AUDIT}[1]{\ensuremath{\mathrm{AUDIT_C}(#1)}}
\newcommand{\SATstar}[1]{\ensuremath{\mathrm{SAT_C^*}(#1)}}
\newcommand{\CsatB}{\ensuremath{\mathrm{SAT_C}(B)}}
\newcommand{\Csat}{\ensuremath{\mathrm{SAT_C}}}
\newcommand{\CsatStar}{\ensuremath{\mathrm{SAT_C^*}}}
\newcommand{\CsatStarB}{\ensuremath{\mathrm{SAT_C^*}(B)}}
\newcommand{\satS}{\ensuremath{\mathrm{SAT_C}(S)}}
\newcommand{\satStarS}{\ensuremath{\mathrm{SAT_C^*}(S)}}
\newcommand{\complexityclassname}[1]{\ensuremath{\mathrm{#1}}}
\newcommand{\DP}{\ensuremath{\mathrm{D}}^\ensuremath{\mathrm{P}}}
\newcommand{\coNP}{\complexityclassname{coNP}}
\renewcommand{\P}{\complexityclassname{P}}
\newcommand{\US}{\complexityclassname{US}}
\newcommand{\CVP}[1]{\ensuremath{\mathrm{CVP}(#1)}}
\newcommand{\SATf}[1]{\ensuremath{\mathrm{SAT}(#1)}}
\newcommand{\SATc}[1]{\ensuremath{\mathrm{SAT}_c(#1)}}
\newcommand{\pc}[1]{\ensuremath{\left[ #1 \right]}}
\newcommand{\enu}[3]{\ensuremath{{#1_{#2}},\allowbreak \dots,\allowbreak#1_{#3}}}
\newcommand\eqd{\eqdef}
\newcommand{\BF}{{\mathrm{BF}}\xspace}
\newcommand{\Mon}{\ensuremath{\mathrm{M}}\xspace}
\newcommand{\Nmon}{\ensuremath{\mathrm{M_1}}\xspace}
\newcommand{\Omon}{\ensuremath{\mathrm{M_0}}\xspace}
\newcommand{\Tmon}{\ensuremath{\mathrm{M_2}}\xspace}
\newcommand{\Self}{\ensuremath{\mathrm{D}}\xspace}
\newcommand{\Nself}{\ensuremath{\mathrm{D_1}}\xspace}
\newcommand{\Tself}{\ensuremath{\mathrm{D_2}}\xspace}
\newcommand{\Nrep}{\ensuremath{\mathrm{R_1}}\xspace}
\newcommand{\Orep}{\ensuremath{\mathrm{R_0}}\xspace}
\newcommand{\Rep}{\ensuremath{\mathrm{R_2}}\xspace}
\newcommand{\VelC}{\ensuremath{\mathrm{V}}\xspace}
\newcommand{\Tvel}{\ensuremath{\mathrm{V_2}}\xspace}
\newcommand{\Ovel}{\ensuremath{\mathrm{V_0}}\xspace}
\newcommand{\Nvel}{\ensuremath{\mathrm{V_1}}\xspace}
\newcommand{\EtC}{\ensuremath{\mathrm{E}}\xspace}
\newcommand{\Oet}{\ensuremath{\mathrm{E_0}}\xspace}
\newcommand{\Net}{\ensuremath{\mathrm{E_1}}\xspace}
\newcommand{\Tet}{\ensuremath{\mathrm{E_2}}\xspace}
\newcommand{\Sep}{{\mathrm{S}}\xspace}
\newcommand{\Osep}{\ensuremath{{\mathrm{S}_0}}\xspace}
\newcommand{\ORsep}{\ensuremath{{\mathrm{S}_{02}}}\xspace}
\newcommand{\OMsep}{\ensuremath{{\mathrm{S}_{01}}}\xspace}
\newcommand{\ORMsep}{\ensuremath{{\mathrm{S}_{00}}}\xspace}
\newcommand{\Nsep}{\ensuremath{{\mathrm{S}_1}}\xspace}
\newcommand{\NRsep}{\ensuremath{{\mathrm{S}_{12}}}\xspace}
\newcommand{\NMsep}{\ensuremath{{\mathrm{S}_{11}}}\xspace}
\newcommand{\NRMsep}{\ensuremath{{\mathrm{S}_{10}}}\xspace}
\newcommand{\Omsep}[1]{\ensuremath{{\mathrm{S}_0^{#1}}}\xspace}
\newcommand{\Nmsep}[1]{\ensuremath{{\mathrm{S}_1^{#1}}}\xspace}
\newcommand{\Lin}{\ensuremath{{\mathrm{L}}}\xspace}
\newcommand{\Olin}{\ensuremath{{\mathrm{L_0}}}\xspace}
\newcommand{\Nlin}{\ensuremath{{\mathrm{L_1}}}\xspace}
\newcommand{\Tlin}{\ensuremath{{\mathrm{L_2}}}\xspace}
\newcommand{\Rlin}{\ensuremath{{\mathrm{L_3}}}\xspace}
\newcommand{\Neg}{\ensuremath{{\mathrm{N}}}\xspace}
\newcommand{\Tneg}{\ensuremath{{\mathrm{N_2}}}\xspace}
\newcommand{\Ids}{\ensuremath{{\mathrm{I}}}\xspace}
\newcommand{\TIds}{\ensuremath{{\mathrm{I_2}}}\xspace}
\newcommand{\OIds}{\ensuremath{{\mathrm{I_0}}}\xspace}
\newcommand{\NIds}{\ensuremath{{\mathrm{I_1}}}\xspace}
\newcommand{\Cons}{\ensuremath{{\mathrm{C}}}\xspace}
\newcommand{\OCons}{\ensuremath{{\mathrm{C_0}}}\xspace}
\newcommand{\NCons}{\ensuremath{{\mathrm{C_1}}}\xspace}
\newcommand{\zvec}{0, \dots ,0}
\newcommand{\ovec}{1, \dots ,1}
\newcommand{\xvec}{x_1, \dots, x_n}
\newcommand{\dua}{{\mathrm{dual}}}
\newcommand{\redpeq}{\ensuremath{\equiv_m^p}\xspace}
\newcommand{\redl}{\ensuremath{\le_m^\mathrm{log}}\xspace}
\newcommand{\redleq}{\ensuremath{\equiv_{\mathrm{m}}^{\mathrm{log}}}}
\newcommand{\loeq}[0]{\ensuremath{\equiv}}
\newcommand{\loiso}[0]{\ensuremath{\cong}}
\newcommand{\PPrefixC}[0]{\ensuremath{\mathrm{C}}}
\newcommand{\parity}{\ensuremath{\oplus}}
\newcommand{\PCEQ}[1]{\ensuremath{\mathrm{EQ_\PPrefixC}(#1)}}
\newcommand{\PCISO}[1]{\ensuremath{\mathrm{ISO_\PPrefixC}(#1)}}
\newcommand{\GAP}{\ensuremath{\mathrm{GAP}}\xspace}
\newcommand{\CGAP}{\ensuremath{\mathrm{\overline{GAP}}}\xspace}
\newcommand{\GOAP}{\ensuremath{\mathrm{GOAP}}\xspace}
\newcommand{\CGOAP}{\ensuremath{\mathrm{\overline{GOAP}}}\xspace}
\newcommand{\PCVAL}[1]{\ensuremath{\mathrm{VAL_\PPrefixC}(#1)}}
\newcommand{\TTAUT}{\ensuremath{3\,\text{-}\TAUT}\xspace}
\newcommand{\PCSAT}[1]{\ensuremath{\mathrm{SAT}_\PPrefixC(#1)}}
\newcommand{\ParL}{\ensuremath{\parity\L}\xspace}
\newcommand{\range}[2]{\ensuremath{#1,\allowbreak \dots ,\allowbreak #2}}
\newcommand{\iffd}{\ensuremath{\Leftrightarrow}} 
\newcommand{\gdw}{\ensuremath{\text{ iff }}\xspace}
\newcommand{\boolf}[2]{\ensuremath{f_{#1}(#2)}}
\newcommand{\TDNF}[0]{\ensuremath{3}\,\text{-DNF}\xspace}
\newcommand{\define}{\ensuremath{=}} 
\newcommand{\NAND}{\ensuremath{{\mathit{nand}}}}
\newcommand{\AND}{\ensuremath{\land}}
\newcommand{\OR}{\ensuremath{\lor}}
\newcommand{\XOR}{\ensuremath{\oplus}}
\newcommand{\NOT}{\ensuremath{\neg}}
\newcommand{\EQ}{\ensuremath{\equiv}}
\newcommand{\IMP}{\ensuremath{\rightarrow}}
\newcommand{\czero}{0}
\newcommand{\cone}{1}
\newcommand{\ID}{\ensuremath{{\mathit{id}}}}
\newcommand{\TAUT}{\ensuremath{\mathrm{TAUT}}}
\newcommand{\EQB}{\ensuremath{\mathrm{EQ}}}
\newcommand{\caseDistinction}[1]
           {\left\{ 
            \begin{array}{l@{\quad}l}
              #1
            \end{array} \right. 
           }
\newcommand{\EFV}[1]{\ensuremath{\exists \ \mathrm{FV_C}(#1)}}
\newcommand{\SATP}{\ensuremath{\mathrm{SATP}}}
\newcommand{\sSATF}[1]{\ensuremath{\mathrm{Select SAT_F}(#1)}}
\newcommand{\countone}[1]{\ensuremath{\#_1(#1)}}
\newcommand{\countzero}[1]{\ensuremath{\#_0(#1)}}
\newcommand{\dual}[1]{\ensuremath{\mathrm{dual}(#1)}}
\newcommand{\uint}{\ensuremath{\mathbb{N}}}
\newcommand{\clone}[1]{\ensuremath{\left[ #1\right]}}
\newcommand{\redpm}{\ensuremath{\leq_{m}^{p}}}
\newcommand{\nmodels}{\ensuremath{\not\models}}
\newcommand{\clonename}[1]{\mathrm{#1}}
\newcommand{\suchthat}{\ensuremath{:}}
\newcommand{\cM}{\ensuremath{\clonename{M}}}
\newcommand{\cR}{\ensuremath{\clonename{R}}}
\newcommand{\cBF}{\ensuremath{\clonename{BF}}}
\newcommand{\cS}{\ensuremath{\clonename{S}}}
\newcommand{\cD}{\ensuremath{\clonename{D}}}
\newcommand{\cV}{\ensuremath{\clonename{V}}}
\newcommand{\cE}{\ensuremath{\clonename{E}}}
\newcommand{\cL}{\ensuremath{\clonename{L}}}
\newcommand{\cI}{\ensuremath{\clonename{I}}}
\newcommand{\cN}{\ensuremath{\clonename{N}}}

%% file: patey.bbl
\begin{thebibliography}{10}

\bibitem{allender2005complexity}
Eric Allender, Michael Bauland, Neil Immerman, Henning Schnoor, and Heribert
  Vollmer.
\newblock The complexity of satisfiability problems: Refining schaefer’s
  theorem.
\newblock In {\em Mathematical Foundations of Computer Science 2005}, pages
  71--82. Springer, 2005.

\bibitem{bienvenurwkl}
Laurent Bienvenu, Ludovic Patey, and Paul Shafer.
\newblock {A Ramsey-Type König's lemma and its variants}.
\newblock in preparation.

\bibitem{bulatov2002dichotomy}
Andrei~A Bulatov.
\newblock A dichotomy theorem for constraints on a three-element set.
\newblock In {\em Proceedings of the 43rd Annual IEEE Symposium on Foundations
  of Computer Science (FOCS'02)}, pages 649--658. IEEE, 2002.

\bibitem{creignou1996complexity}
Nadia Creignou and Miki Hermann.
\newblock Complexity of generalized satisfiability counting problems.
\newblock {\em Information and Computation}, 125(1):1--12, 1996.

\bibitem{flood2012reverse}
Stephen Flood.
\newblock {Reverse mathematics and a Ramsey-type K{\"o}nig's Lemma}.
\newblock {\em Journal of Symbolic Logic}, 77(4):1272--1280, 2012.

\bibitem{hirst1990marriage}
Jeffry~L. Hirst.
\newblock Marriage theorems and reverse mathematics.
\newblock In {\em Logic and computation ({P}ittsburgh, {PA}, 1987)}, volume 106
  of {\em Contemp. Math.}, pages 181--196. Amer. Math. Soc., Providence, RI,
  1990.

\bibitem{khanna1996optimization}
Sanjeev Khanna and Madhu Sudan.
\newblock The optimization complexity of constraint satisfaction problems.
\newblock In {\em Electonic Colloquium on Computational Complexity}. Citeseer,
  1996.

\bibitem{marek2009complexity}
Victor~W Marek and Jeffrey~B Remmel.
\newblock The complexity of recursive constraint satisfaction problems.
\newblock {\em Annals of Pure and Applied Logic}, 161(3):447--457, 2009.

\bibitem{post1942two}
Emil~L Post.
\newblock {\em The two-valued iterative systems of mathematical logic}.
\newblock Number~5. Princeton University Press, 1942.

\bibitem{schaefer1978complexity}
Thomas~J. Schaefer.
\newblock The complexity of satisfiability problems.
\newblock In {\em Conference {R}ecord of the {T}enth {A}nnual {ACM} {S}ymposium
  on {T}heory of {C}omputing ({S}an {D}iego, {C}alif., 1978)}, pages 216--226.
  ACM, New York, 1978.

\bibitem{simpson2009subsystems}
Stephen~G. Simpson.
\newblock {\em Subsystems of second order arithmetic}.
\newblock Perspectives in Logic. Cambridge University Press, Cambridge;
  Association for Symbolic Logic, Poughkeepsie, NY, second edition, 2009.

\end{thebibliography}
